\documentclass[a4paper,12pt]{article}
\usepackage{preamble}

\bibliography{references}

\title{The Constant Domain Axiom in Toposes}
\date{\today}
\author{Jérémie Marquès}

\begin{document}

\maketitle

\begin{abstract}
	Constant domain intuitionistic logic admits a complete semantics in presheaf toposes, by interpreting sorts as constant presheaves and predicates as arbitrary sub-presheaves. The goal of this note is to point out how this fits in topos theory, replacing constant presheaves with objects that are covert and Hausdorff when considered as discrete locales. We call these objects ``CD'' and we show that they form a Boolean pretopos in any topos.
\end{abstract}

\section{Constant domain logic}

Constant domain logic is obtained by adding the \emph{constant domain principle} to intuitionistic logic:
\begin{equation}\label{eq:cofrob}\tag{CD}
	∀a:[φ(a)∨ψ] ⇒ [∀a:φ(a)] ∨ ψ
\end{equation}
It is dual to the \emph{Frobenius law} which, on the other hand, is part of standard intuitionistic logic:
\begin{equation}\label{eq:frob}\tag{Frob}
	[∃a:φ(a)] ∧ ψ ⇒ ∃a:[φ(a)∧ψ]
\end{equation}
Hence, \eqref{eq:cofrob} makes intuitionistic logic a bit more symmetrical. This is reflected in the standard presheaf semantics of constant domain logic, where sorts are interpreted as constant presheaves. A universal quantification is then interpreted as a quantification over the individuals at the \emph{current stage}, just like for existential quantification. On the other hand, implication still requires us to look at the future stages.

The reader can find more background on constant domain logic in the introduction of \cite{MinOlkUrq2013}, where it is shown that it fails to satisfy the Craig interpolation property.

At first, the constant domain modification appears a bit artificial in categorical logic. The most striking issue is that we cannot force \eqref{eq:cofrob} to hold ``uniformly'' when quantification ranges over the fibers of any map, or else implication would also be impacted and we would end up with Boolean logic. In particular, if \eqref{eq:cofrob} holds uniformly in a topos, then it is Boolean. We need a distinguished class of objects that can be used as the domain of quantification in \eqref{eq:cofrob}.

Nonetheless, constant domain logic is connected to standard concepts in locale theory. Some background can be found in \cite[C]{Johnstone2002v2}. In the classical semantics of intuitionistic logic in a topos, sorts and formulas are interpreted as objects and subobjects. We can place ourselves in a more general setting where sorts are interpreted as locales and formulas as open sublocales. The Frobenius law \eqref{eq:frob} defines \emph{open maps} of locales. It is therefore natural to interpret every map as an open map. Equivalently, every sort is interpreted as an object $X$ such that $X→1$ and $X→X×X$ are open, which means that $X$ is discrete. Symmetrically, the constant domain law \eqref{eq:cofrob} defines \emph{closed maps} of locales. Thus every map over which \eqref{eq:cofrob} applies should be interpreted as a closed map of locales. Similarly to the situation with open maps, it is equivalent to requiring that $X→1$ and $X→X×X$ are closed for every sort $X$, as shown in Proposition~\ref{prop:stab-CD} below.

\section{Constant domain objects}

We start by specializing some localic terminology to discrete spaces. A morphism $f : X→Y$ in a topos is \emph{closed} if \eqref{eq:cofrob} holds when the quantifier is understood to range over an arbitrary fiber of $f$. Equivalently, this means that $f$ is closed as a map of discrete locales. An object $X$ is \emph{covert} if the map $X→1$ is closed or if, equivalently, $X$-indexed infima distribute over finite suprema in the subobject classifier. We say that $X$ is \emph{Hausdorff} if the map $X→X×X$ is closed.

Information about closed maps can be found in \cite[C3.2]{Johnstone2002v2}. It is pointed out there that properness (relative compactness) is better behaved as a dual of openness. This seems unfortunately too strong to offer a complete semantics for constant domain logic, because a theory could force the existence of infinitely many distinct constants.

Note that a discrete $X$ is Hausdorff if and only if it is decidable in the sense that equality on $X$ is complemented. More generally, a monomorphism $X↪Y$ is closed if and only if $X$ is a complemented subobject of $Y$: the complement of $X↪Y$ is an open sublocale, hence a subobject of $Y$. This gives a topological explanation of the fact that equality is decidable in constant domain logic. We will say ``decidable'' instead of ``Hausdorff.''

A \emph{CD object} in a topos is an object that is covert and decidable (``CD'' could also stand for ``constant domain'').

\begin{prop}{}{stab-CD}
	The CD objects in a topos form a Boolean pretopos closed under taking complemented subobjects. Moreover, any map between CD objects is closed.
\end{prop}

\begin{proof*}{}
	We start by showing that any map between CD objects is closed. Let $f : X→Y$ be such a map and let $φ ⊆ X$. Since $Y$ is decidable,
	\[
		y ∈ ∀_f(φ) ⇔ ∀x∈X : (f(x)=y) → φ(x) ⇔ ∀x∈X : (f(x)≠y) ∨ φ(x) \text{.}
	\]
	Let $ψ$ be a proposition. Using that $X$ is covert,
	\[ ∀x∈X : [(f(x)≠y)∨φ(x)∨ψ] ⇔ [∀x∈X : (f(x)≠y)∨φ(x)]∨ψ \]
	and this shows that $f$ is closed.
	
	We show that CD objects are stable under taking complemented subobjects. Let $X$ be CD and let $φ ⊆ X$ be complemented. Then the composite $φ → X → 1$ is closed as a composite of closed maps. Moreover, $φ$ is decidable as a subobject of a decidable object.
	
	Next, we show that CD objects are stable under finite limits. Let $X$ and $Y$ be CD objects. We will use that the product of a closed map by a discrete locale is closed \cite[Lem.~C3.2.3]{Johnstone2002v2}. The product $X×Y$ is decidable since its diagonal is the composite of closed maps
	\[ X×Y → X^2×Y → X^2×Y^2 \text{.} \]
	Alternatively, the conjunction of the two complemented equalities is complemented. To see that $X×Y$ is covert, we compose the closed maps $X×Y→X→1$. Note also that the terminal object $1$ is CD.
	
	Let $f, g : X⇉Y$ be a pair of maps between CD objects. Since $Y$ is decidable, the equalizer $\eq(f,g) ⊆ X$ is complemented, hence CD. This concludes the proof that CD objects are stable under finite limits.
	
	The disjoint sum of two CD objects $X$ and $Y$ is again CD: Decidability of $X+Y$ is easy to see. It is covert because if $C_1+C_2 ⊆ X+Y$ is a closed sublocale, then its image in $1$ is the union of the images of $C_1$ and $C_2$ which is closed. The empty object is also CD.
	
	If $f : X→Y$ is a map between CD objects, then $f$ is closed and in particular $f[X] ⊆ Y$ is closed, which means that it is complemented and thus also CD. This shows that the CD objects form a coherent category.
	
	Finally, we show that CD objects are closed under quotients by decidable equivalence relations. Let $R ⊆ X^2$ be a decidable equivalence relation on a CD object. Then $X/R$ is decidable because $R$ is. It is also covert because the composite $X↠X/R→1$ is closed and $X↠X/R$ is surjective.
\end{proof*}

Constant domain logic can be interpreted in any topos as long as the sorts are interpreted as CD objects. The following proposition, which can be obtained from \cite[Lem.~C3.2.4]{Johnstone2002v2}, shows that we recover the usual constant domain semantics in presheaf toposes. In general, the relevance of this notion is unclear to me.

\begin{prop}{}{}
	A presheaf $F : \cC → \cSet$ is covert in $[\cC,\cSet]$ if and only if it is valued in surjective maps. It is CD if and only if it is valued in bijective maps.
\end{prop}

\paragraph{Acknowledgments} I thank Sam van Gool and Dominik Kirst for the interesting discussion which led to this note.

\includebibliography

\end{document}